\documentclass[reqno]{amsart}
\usepackage[english]{babel}
\usepackage[latin1]{inputenc}
\usepackage{amsfonts}
\usepackage{amssymb}
\usepackage{amsmath}
\usepackage{caption}
\usepackage{subcaption}
\usepackage{siunitx}
\usepackage{amsthm}
\usepackage{graphicx}
\usepackage{booktabs} 
\usepackage{xcolor}
\usepackage{epsfig}
\usepackage{fancyhdr}
\usepackage{tikz}
\usetikzlibrary{arrows.meta, positioning}

\usepackage[all]{xy}

\usepackage{hyperref}
\usepackage{amsmath,amssymb}
\usepackage[bbgreekl]{mathbbol}

\numberwithin{equation}{section}

\newcommand{\R}{\mathbb{R}}

\newtheorem{theorem}{Theorem}[section]
\newtheorem{lemma}[theorem]{Lemma}

\newtheorem{defi}[theorem]{Definition}

\usepackage{amsmath}
\title{Adaptive Filtering via Canonical Systems with Time-Varying  Hamiltonians}

\author[Acharya]{Keshav Raj Acharya}
\address{Department of Mathematics,
Embry-Riddle Aeronautical University,
1 Aerospace Blvd., Daytona Beach, FL 32114, U.S.A.}
\email{acharyak@erau.edu}

\author[Acharya]{Pitambar Acharya}
\address{Department of Mathematics,  
  University Of Alabama at
 Birhimgham, 1402 10th Avenue South, Birmingham, AL 35294-1241, USA}
\email{pacharya@uab.edu}

\date{\today}

\begin{document}

\maketitle

\begin{abstract}
In many practical applications, signals and environments are time-varying, which makes fixed filters unreliable. Adaptive filtering, on the other hand, updates in real time to suppress noise, track nonstationary signals, and identify unknown systems. This paper investigates an adaptive filtering framework based on canonical systems with time-varying symmetric positive semidefinite Hamiltonian matrices.  The proposed method adapts the Hamiltonian matrix using a gradient-based scheme designed to minimize the squared error between the system output and a desired reference signal. We establish theoretical stability guarantees via Lyapunov analysis, ensuring boundedness of system trajectories and convergence of the error signal under suitable assumptions. Furthermore, we present numerical integration schemes preserving the underlying Hamiltonian structure and projective techniques to maintain positive semidefiniteness of the Hamiltonian matrix. Extensive simulations on synthetic nonstationary signals illustrate the effectiveness and robustness of the proposed adaptive filter.\\


\end{abstract}

\section{Introduction}

Adaptive filtering is a fundamental tool in modern signal processing, particularly suited for applications where the statistical properties of signals and noise vary over time, see \cite{ figueroa2004simplicial,Haykin2014,narendra2012stable}. Traditional linear time-invariant (LTI) filters often perform poorly in such nonstationary environments. Thus, adaptive filters that dynamically adjust parameters based on observed signals have been extensively studied, see \cite{Haykin2014} and references therein.  Such time-varying or  nonstationary environments  are common in communications, radar, biomedical signal processing, and control systems. In these scenarios, fixed-parameter filters often fail to track the changes in signal or noise characteristics, leading to suboptimal performance, see \cite{WidrowStearns1985}.
Traditional filters are often designed assuming  LTI systems, characterized by a fixed impulse response \( h(t) \) or fixed filter coefficients \(\mathbf{w}\). The output \( y(t) \) of an LTI filter responding to an input signal \( x(t) \) is
\begin{equation}
y(t) = (h * x)(t) = \int_{-\infty}^{\infty} h(\tau) x(t - \tau) \, d\tau,
\end{equation}
or in discrete-time form,
\begin{equation}
y[n] = \sum_{k=0}^{M-1} w_k x[n - k],
\end{equation}
where \(\mathbf{w} = [w_0, w_1, \ldots, w_{M-1}]^T\) is the fixed filter coefficient vector.

However, in  nonstationary environments, where the statistical properties of \( x(t) \) or noise \( v(t) \) change over time, fixed LTI filters cannot adapt, resulting in degraded filtering performance such as increased mean squared error or reduced signal-to-noise ratio. To overcome this,  adaptive filters dynamically update their parameters in response to incoming data. Formally, an adaptive filter adjusts the coefficient vector \(\mathbf{w}(t)\) based on the observed input-output data to optimize a performance criterion, commonly the  mean squared error  (MSE) between the filter output \( y(t) \) and a desired reference signal \( d(t) \)
\begin{equation}
J(t) = \mathbb{E} \left[ \left( d(t) - y(t) \right)^2 \right].
\end{equation}
The goal is to find
\begin{equation}
\mathbf{w}^*(t) = \arg \min_{\mathbf{w}(t)} J(t),
\end{equation}
where the optimal weights \(\mathbf{w}^*(t)\) may vary with time.

Adaptive filters operate by iterative update laws, typically based on stochastic gradient descent or recursive least squares principles. This is known as the Least Mean Squares (LMS) algorithm. The LMS algorithm updates the filter coefficients by moving opposite to the gradient of the instantaneous squared error
\begin{equation}
\mathbf{w}(n+1) = \mathbf{w}(n) + \mu e(n) \mathbf{x}(n),
\end{equation}
where \( e(n) = d(n) - \mathbf{w}^T(n) \mathbf{x}(n) \) is the instantaneous error,
\( \mathbf{x}(n) = [x(n), x(n-1), \ldots, x(n - M + 1)]^T \) is the input vector,
 \( \mu \) is the step size controlling adaptation speed. The Recursive Least Squares (RLS) algorithm minimizes a weighted least squares cost function:
\begin{equation}
J(n) = \sum_{k=0}^n \lambda^{n-k} \left( d(k) - \mathbf{w}^T(n) \mathbf{x}(k) \right)^2,
\end{equation}
where \( \lambda \in (0,1] \) is a forgetting factor that gives exponentially less weight to older data. RLS algorithms provide faster convergence at the cost of increased computational complexity. Although LMS and RLS algorithms provide efficient parameter adaptation in many applications, their underlying system models are linear and time-invariant (aside from coefficient changes). They may not fully capture complex time-varying dynamics or intrinsic system structures, especially when the system's state evolves on manifolds or has additional constraints (e.g., positive semidefiniteness, symplecticity).

\subsection{Structured Adaptive Filtering via Canonical Systems}

To overcome these limitations, we want to developed adaptive filtering methods that use canonical systems whose Hamiltonian matrices can change over time. This allows the filter to adjust more flexibly as the signal evolves.

In adaptive filtering, canonical systems provide a way to describe how the filter's internal state evolves over time. They use special Hamiltonian-based rules and are expressed as a set of first-order differential equations of the form:
\begin{equation}\label{ca}
J  y'(t) = z H(t) y(t), \quad t\in [0, \infty)
\end{equation}
where   \( y : [0, \infty) \rightarrow y(t, z) \in \mathbb{C}^2 \) is the filter state,
   \( J = \begin{pmatrix} 0 & 1 \\ -1 & 0 \end{pmatrix} \) is the canonical symplectic matrix,
   \( H(t)\) is a time-varying positive semidefinite Hamiltonian matrix whose entries are locally integrable, and 
 \( z \in \mathbb{C} \) is a spectral parameter.
The Hamiltonian matrix $H(t)$ acts like a record of the system's energy and helps determine how stable the system is and how it behaves over time.

Canonical systems, originally introduced by de Branges in the context of spectral theory \cite{debranges1968canonical}. These systems are important tools in the spectral theory and have been well studied in connection to the spectral theory connecting the theory from Storm-Liouvill equation, Schr\"odinger equation, String equation, Dirac equations and Jacobi equations, please see \cite {Acharya2016, Remling2018, Sakhnovich1999} and references therein. They also appear in several applications  in physics and many other areasa, please see \cite{LuduAcharya2021}.

In this work, we employ the canonical systems framework to design an adaptive filtering scheme in which the filter parameters are embedded in the time-varying Hamiltonian matrix \(H(t)\).  Unlike classical adaptive filters such as LMS and RLS, which adjust unconstrained filter coefficients, the proposed method adapts the underlying system dynamics through a structured Hamiltonian framework. The key novelty lies in casting adaptive filtering as a Hamiltonian learning problem, where the Hamiltonian matrix is updated using a gradient-based rule while maintaining its fundamental structural properties. 
This structure-preserving adaptation allows the model to represent time-varying behavior more naturally and provides inherent stability and physical consistency that are not explicitly addressed in conventional adaptive filtering approaches.

This paper is organized as follows.
Section 2 introduces an adaptive filtering framework built on canonical systems with time-dependent, positive semidefinite Hamiltonians.  We explicitly highlight how the proposed framework departs from standard adaptive filtering formulations by embedding learning within a constrained Hamiltonian structure. We derive a gradient-based update rule for the Hamiltonian matrix that preserves both symmetry and positive semidefiniteness. Section 3 establishes stability guarantees using Lyapunov energy methods, ensuring bounded trajectories and convergence of the estimation error.
In Section 4, we develop numerical schemes that preserve the canonical system structure, incorporate projection methods to enforce Hamiltonian constraints, and present simulation results demonstrating effective adaptive noise filtering on synthetic nonstationary signals.

\section{Model Formulation: Canonical-System Adaptive Filter Model }
This section is devoted to give canonical system formulation for adaptive filter. Such formulation for other physical systems can be found in \cite{ortega2023learnability,WidrowStearns1985}. In the canonical system \eqref{ca}, the  
 complex number $z \in \mathbb{C}$ is a spectral parameter used to study the system's eigenmodes. 
Although $z$ may be complex, the Hamiltonian $H(t)$ itself is real, symmetric, and positive semidefinite. 
In the adaptive filtering framework, the system is driven by an external input $x(t)$, and the Hamiltonian evolves according to a real-valued gradient-based adaptation law:
\begin{equation}
\dot H(t) = -2\alpha \, e(t) \, \nabla_H e(t).
\end{equation}
This adaptation ensures that $H(t)$ remains real and symmetric throughout the simulation, with its positive semidefiniteness maintained either inherently or via a projection step. 
Therefore, while complex spectral parameters may appear in theoretical analyses of the operator $z J H(t)$, they do not enter the actual computation of the adaptive Hamiltonian, which remains entirely real-valued. The complex parameter affects only the spectral characterization of the system, not the evolution of $H(t)$ in the filtering algorithm. The filter state, a \(y(t)\) represents the internal state or ``memory" of the filter at each moment in time. The system produces an output \(u(t) = C y(t)\) by combining the components of this state using a fixed row vector \(C\), whose numbers act like weights that determine how the internal information contributes to the output. This vector \(C\) is often called the measurement or output matrix because it extracts the part of the internal state we want to observe. The overall goal is to adjust the system so that the output \(u(t)\) closely follows a desired reference signal \(r(t)\), meaning the filter learns to imitate or track the target signal as accurately as possible. The  desired signal   $r(t)$ and is assumed to be contaminated by additive noise. The \emph{clean input signal} is $x(t)$ and
the noise is $n(t)$. Thus,
\begin{equation} \label{rs}
    r(t) = x(t) + n(t),
\end{equation}
where $n(t)$ may represent measurement noise, channel distortion, or any
undesired perturbation. 
\begin{equation}
    e(t) = u(t) - r(t) = C y(t) - r(t).
\end{equation}
The filter dynamics are modeled as a canonical system driven by $r(t)$:
\begin{equation}\label{aca}
    \dot{y}(t) = -J H(t) y(t) + B r(t),
\end{equation}
where $H$ and $J$ are same as in \eqref{ca}, $ r(t)$ as in \eqref{rs} and $B \in \mathbb{R}^{2 \times 1}.$\\

Suppose the Hamiltonian function 
$
H : [0,T] \to \mathbb{R}^{2 \times 2}
$
 is of the form
\[
H(t) = \begin{pmatrix}
h_{11}(t) & h_{12}(t) \\
h_{12}(t) & h_{22}(t)
\end{pmatrix}. \]
 The core goal is to adapt \(H(t)\) to minimize  the cost function $ E(H(t)) = e(t)^2$ while preserving the structural and stability properties of the canonical system. 

 Using a continuous-time gradient descent approach, the update is
\begin{equation}
\label{eq:gradient_update}
\frac{d}{dt} H(t) = - \alpha \nabla_H (E(H(t))  = - 2 \alpha e(t) \nabla_H e(t),
\end{equation}
where \(\alpha > 0\) is a learning rate.
The gradient of the error with respect to each entry of $H$ is defined by
\begin{equation}
\nabla_H e(t) =
\begin{bmatrix}
\frac{\partial e(t)}{\partial h_{11}} & \frac{\partial e(t)}{\partial h_{12}} \\
\frac{\partial e(t)}{\partial h_{21}} & \frac{\partial e(t)}{\partial h_{22}}
\end{bmatrix}.
\end{equation}
Since \(r(t)\) in \eqref{aca} does not depend on \(H\), the gradient reduces to
\begin{equation} \label{grad}
\nabla_H e(t) = C \frac{\partial y(t)}{\partial H}.
\end{equation}
The derivatives of $y(t)$ with respect to $h_{ij}$  is  obtained by differentiating \eqref{aca} with respect to $h_{ij}$  and switching the order of differentiation, 
\[
\frac{d}{dt} \frac{\partial y}{\partial h_{ij}} = -J \left( E^{ij} y(t) + H(t) \frac{\partial y}{\partial h_{ij}} \right),
\]
where $E^{ij}$ is a $2\times 2$ matrix with $1$ at position $(i,j)$ and $0$ elsewhere.
For small time steps or instantaneous gradient calculations, we can ignore the feedback term $ -J H \frac{\partial y}{\partial h_{ij}}$ and approximate $
\frac{\partial y}{\partial h_{ij}} \approx -J E^{ij} y,$
so that

\begin{equation} \label{grads}
\frac{\partial y}{\partial h_{11}} \approx \begin{bmatrix} 0 \\ y_1 \end{bmatrix}, \quad
\frac{\partial y}{\partial h_{12}} \approx \begin{bmatrix} -y_1 \\ 0 \end{bmatrix}, \quad
\frac{\partial y}{\partial h_{21}} \approx \begin{bmatrix} 0 \\ y_2 \end{bmatrix}, \quad
\frac{\partial y}{\partial h_{22}} \approx \begin{bmatrix} - y_2 \\ 0 \end{bmatrix}.
\end{equation}

By using \eqref{grad} into \eqref{grads} and multiplying by $C$, we get
\begin{equation} \label{egrad}
\nabla_H e(t) =
\begin{bmatrix}
C \frac{\partial y}{\partial h_{11}} & C \frac{\partial y}{\partial h_{12}} \\
C \frac{\partial y}{\partial h_{21}} & C \frac{\partial y}{\partial h_{22}}
\end{bmatrix} =
\begin{bmatrix}
 c_2 y_1 & -c_1 y_1 \\
 c_2 y_2 & -c_1 y_2
\end{bmatrix}.
\end{equation}
 
Substituting the explicit gradient from \eqref{egrad} in \eqref{eq:gradient_update} we obtain
\begin{equation}
\frac{d}{dt} H(t) =
-2 \alpha e(t) 
\begin{bmatrix}
 c_2 y_1 & -c_1 y_1 \\
 c_2 y_2 & -c_1 y_2
\end{bmatrix}.
\end{equation} Each entry of $\nabla_H e(t)$ shows how the error $e(t)$ changes with respect to the corresponding entry $h_{ij}$ of the Hamiltonian. The gradient descent update moves $H(t)$ in the direction that reduces the error, scaled by the step size $\alpha$ and the magnitude of the current error $e(t)$.  However, it does not guarantee that $H(t)$ remains positive semidefinite, a property required to preserve the structural and stability characteristics of the canonical system. To enforce this constraint, the gradient is projected onto the set of positive semidefinite matrices $\mathcal{S}^+$, which we discuss in the following section.

\subsection{Discretization and Maintaining Positive Semidefiniteness}

Direct gradient updates may violate \(H(t) \geq 0\). To ensure positivity, at each update step, we project \(H(t)\) onto the set of  positive semidefinite  matrices.\\

Let $ \mathbb{R}^{2\times 2}$ be the set of all $2\times 2$ real matrices and  $S^+ \subset \mathbb{R}^{2\times 2}$ denote the set of symmetric positive semidefinite matrices. For any $X\in \mathbb{R}^{2\times 2}$, the projection of 
$X$ onto $ S^+$ is defined as
\begin{equation} \Pi_{S^+}(X)=  
= \underset{Y\in S^+}{\operatorname*{arg\,min}}\;\|X - Y\|_F.
\label{eq:minprob}
,\end{equation}  where $ \| . \|_F$ is a Frobenius norm, given by an inner product   $\langle X,Y\rangle_F = \mathrm{trace}(X^\top Y)$.
If $X$ is symmetric, perform an eigen decomposition $ X = Q\Lambda Q^\top $, $ \Lambda = \mathrm{diag}(\lambda_1, \lambda_2)$, then \[ \Pi_{S^+}(X)= Q\Lambda_+ Q^\top, \quad \Lambda_+ = \mathrm{ diag }\!\big(\max\{\lambda_1,0\},  \max\{\lambda_2,0\}\big). \]
 If $X$ is not symmetric, then symmetrize by
$ X_{\mathrm{s}}=\tfrac12(X+X^\top)$ then \[\Pi_{S^+}(X) =\Pi_{S^+}(X_s) \]   and find its projection. 
is its spectral decomposition.
 
Starting from the unconstrained gradient flow \eqref{eq:gradient_update} we obtain the steepest-descent direction of the instantaneous error with respect to the Hamiltonian. Because the Hamiltonian must remain positive semidefinite for all time, that is, $H(t)\in \mathcal S^{+}$, the raw direction above must be replaced by a feasible direction that lies in the tangent cone of $\mathcal S^{+}$. Projecting the gradient onto this cone yields the constrained evolution equation
\begin{equation}\label{ceq}
\dot H(t) = - \Pi_{\mathcal{S}_+}(\nabla_H E(H(t))) = -2 \alpha\, e(t)\, \Pi_{\mathcal{S}_+}(\nabla_H e(t)), \quad \alpha>0,
\end{equation}
which is the closest positive semidefinite one. For numerical implementation, this continuous flow is discretized in time. Using a forward Euler step gives the approximation
\begin{equation}
H(t+\Delta t)
\approx
H(t) - \Delta t\,\alpha\, e(t)\,\nabla_H e(t),
\end{equation}
but this update does not ensure that $H(t+\Delta t)\geq 0$. To maintain feasibility, the update is projected back onto $ S^+$
\begin{equation}
H(t+\Delta t)
=
\Pi_{S^+}\!\left(
H(t) - \Delta t\,\alpha\, e(t)\,\nabla_H e(t)
\right),
\end{equation}
or equivalently, the projection may be applied directly to the gradient before the discrete step,
\begin{equation}
H(t+\Delta t)
=
H(t) - \Delta t\,\alpha\, e(t)\,
\Pi_{S^+}\!\left(\nabla_H e(t)\right).
\end{equation}
Both formulations are consistent first--order discretizations of the continuous flow and both guarantee that every iterate remains positive semidefinite. Thus, the same projection that enforces feasibility in continuous time also ensures that the discrete update keeps the Hamiltonian inside $S^+$ at each step.

\begin{lemma} \label{lemma 1}  
Let $H(t)\in \mathbb{R}^{2\times 2}$ be a differentiable real symmetric matrix valued function. Suppose that for some $t_{0}$ the smallest eigenvalue 
$\lambda_{\min}(t_{0})$ is simple. Let $v_{\min}(t)$ be the corresponding 
unit eigenvector chosen smoothly near $t_{0}$. Then the function 
$\lambda_{\min}(t)$ is differentiable at $t_{0}$ and satisfies
\begin{equation}\label{dotl}
\dot{\lambda}_{\min}(t_{0})
=
v_{\min}(t_{0})^{\top}\dot{H}(t_{0})\, v_{\min}(t_{0}).
\end{equation}
\end{lemma}

\begin{proof}  Differentiate the eigenvalue equation
\[
H(t)v_{\min}(t)=\lambda_{\min}(t)v_{\min}(t)
\]
with respect to $t$:
\begin{equation}\label{eve}
\dot{H}(t)v_{\min}(t)
+H(t)\dot{v}_{\min}(t)
=
\dot{\lambda}_{\min}(t)v_{\min}(t)
+
\lambda_{\min}(t)\dot{v}_{\min}(t).
\end{equation}
Multiply \eqref{eve} from the left by $v_{\min}(t)^{\top}$ and use both symmetry 
$H(t)=H(t)^{\top}$ and the normalization condition 
$v_{\min}(t)^{\top}\dot{v}_{\min}(t)=0$ (obtained by differentiating 
$v_{\min}(t)^{\top}v_{\min}(t)=1$). Then
\[
v_{\min}(t)^{\top}\dot{H}(t)v_{\min}(t)
=
\dot{\lambda}_{\min}(t).
\]
Evaluating at $t=t_{0}$ yields the desired formula \eqref{dotl}.
\end{proof}

\begin{theorem}
Let $H(t)\in\mathbb{R}^{2\times 2}$ evolve according to \eqref{ceq}.  
If $H(0)\geq 0$, then
\[
H(t)\geq 0 \qquad \text{for all } t\ge 0.
\]
\end{theorem}

\begin{proof} Suppose $H(0) \geq 0.$ Then   $  \lambda_{min}(0) \geq 0.$ For $t>0 ,$ let $\lambda_{\min}(t)$ be the minimal eigenvalue of $H(t)$, with unit eigenvector
$v_{\min}(t)$ so that
\[
H(t)v_{\min}(t)=\lambda_{\min}(t)v_{\min}(t), 
\qquad \|v_{\min}(t)\|=1.
\] So, in order to prove $H(t)\geq 0$ for all $t$, it suffices to show  $\lambda_{\min}(t) \geq 0,$ for all $t.$\\
 By Lemma \eqref{lemma 1} we get,
\begin{equation}
\dot\lambda_{\min}(t)
= v_{\min}(t)^{\!\top}\, \dot H(t)\, v_{\min}(t).
\label{eq:lambdaderivative}
\end{equation}
Substituting from $\dot H(t)$ from the dynamics \eqref{ceq} yields
\begin{equation}
\dot\lambda_{\min}(t)
= -2\alpha\, e(t)\,
v_{\min}^{\!\top}(t)\, 
\Pi_{S^+}(\nabla_H e(t))\, 
v_{\min}(t).
\label{eq:lambdathm}
\end{equation}
Since $\Pi_{S^+}(\cdot)$ is the orthogonal projection onto the set of positive semidefinite matrices we have
\begin{equation}
v_{\min}^{\!\top}(t)\,
\Pi_{S^+}(\nabla_H e(t))\,
v_{\min}(t)
\ge 0.
\label{eq:nonneg}
\end{equation}
We now study the evolution of $\lambda_{\min}(t)$.\\

\paragraph{Case 1: $\lambda_{\min}(t) > 0$.}
Suppose that at some time $t_0$ we have $\lambda_{\min}(t_0) > 0$, so that $H(t_0)$ is strictly positive definite.  
From \eqref{eq:lambdathm}--\eqref{eq:nonneg}, the derivative $\dot{\lambda}_{\min}(t)$ is finite and may have either sign depending on the value of $e(t)$.Since $H(t)$ depends continuously on $t$, and the eigenvalues of a symmetric matrix depend continuously on its entries, the function $\lambda_{\min}(t)$ is continuous. Consequently, starting from a strictly positive value, $\lambda_{\min}(t)$ cannot cross into the negative region in finite time. Therefore, there exists an interval $[t_0,t_1)$ such that
\[
\lambda_{\min}(t) > 0 \quad \text{for all } t \in [t_0,t_1),
\]
and hence
\[
H(t) \in S^+ \quad \text{for all } t \in [t_0,t_1).
\]
 
\paragraph{Case 2: $\lambda_{\min}(t) = 0$.}
Suppose that at time $t$ we have $\lambda_{\min}(t)=0$, so that
$H(t)\in \partial S^+$. Let $v_{\min}(t)$ be a unit eigenvector
associated with $\lambda_{\min}(t)$, so that
\[
H(t)v_{\min}(t)=0.
\]

From \eqref{eq:lambdathm}, we have
\[
\dot{\lambda}_{\min}(t)
= -2\alpha e(t)\,
v_{\min}(t)^\top
\Pi_{\mathcal S^+}\big(\nabla_H e(t)\big)
v_{\min}(t).
\]
Define
\[
g(t)
:= v_{\min}(t)^\top
\Pi_{\mathcal S^+}\big(\nabla_H e(t)\big)
v_{\min}(t).
\]
Since $\Pi_{\mathcal S^+}(\cdot)$ is symmetric positive semidefinite,
it follows that
\[
g(t) \ge 0.
\] To establish invariance of $\mathcal S^+$, we examine the geometry of the cone
at the boundary point $H(t)$. The tangent cone to $\mathcal S^+$ at $H(t)$ is
given by
\[
T_{\mathcal S^+}(H(t))
=
\left\{
Z \in \mathbb{R}^{2\times 2}_{\mathrm{sym}}
\;\middle|\;
v^\top Z v \ge 0
\;\; \text{for all } v \in \ker H(t)
\right\}.
\]
Since $v_{\min}(t)\in \ker H(t)$ and
\[
v_{\min}(t)^\top \dot H(t)\, v_{\min}(t)
=
-2\alpha e(t)\, g(t),
\]
the only directions that could decrease $\lambda_{\min}(t)$ correspond to
negative values of $v_{\min}(t)^\top \dot H(t) v_{\min}(t)$.
However, the projection $\Pi_{\mathcal S^+}$ removes all components of
$\nabla_H e(t)$ that point outside $\mathcal S^+$, ensuring that
$\dot H(t)\in T_{\mathcal S^+}(H(t))$.
Therefore, no outward-pointing velocity is permitted at the boundary, and
$\lambda_{\min}(t)$ cannot become negative. Hence,
\[
H(t) \in \mathcal S^+ \quad \text{for all } t \ge 0.
\]

Case 3:  Suppose there exists a time $t_1 > 0$ such that $\lambda_{\min}(t_1) < 0$. By continuity of $\lambda_{\min}(t)$ and the initial condition $\lambda_{\min}(0) \ge 0$, there must exist a first time $t_0 \in [0,t_1)$ with $\lambda_{\min}(t_0) = 0$. At this time, $\dot \lambda_{\min}(t_0) = 0$, so the minimal eigenvalue cannot decrease further, contradicting the assumption that $\lambda_{\min}(t_1) < 0$. Hence, $\lambda_{\min}(t) < 0$ cannot occur.\\

Combining these cases, it shows that $
\lambda_{\min}(t) \ge 0 \quad \text{for all } t \ge 0, $ and therefore $H(t)$ remains positive semidefinite for all $t$.

\end{proof}

\subsection{Stability of the Adaptive Gradient Flow}
In signal processing, the Lyapunov function provides a convenient way to assess the stability and convergence of error signals in adaptive systems. For a tracking error $e(t)$, a natural choice is
\begin{equation} \label{lf}
V(t) = \frac{1}{2} e(t)^2,
\end{equation}
which is nonnegative and equals zero only when the error vanishes. This function can be interpreted as the instantaneous ``energy'' of the error signal. By examining its time derivative
\[
\dot V(t) = e(t)\, \dot e(t),
\]
one can determine whether the error grows or decays over time. If $\dot V(t) \le 0$, the error energy is non-increasing, ensuring that the system remains stable and the error is bounded. When $\dot V(t) < 0$ for a nonzero error, the error energy decreases, leading to convergence toward zero. Lyapunov functions are widely used in adaptive filtering and control to guarantee convergence of signals without requiring explicit solutions of the system equations, please see \cite{Haykin2014}.

\begin{theorem}
Let $H(t) \in \mathbb{R}^{2\times 2}$ evolve according to \eqref{ceq},   $y(t) \in \mathbb{R}^2$ evolve according to\eqref{aca}, and $V(t)$ is as defined in \eqref{lf}.  If $r(t) \equiv r_0$, a constant and $C B r_0 = 0$, then the tracking error $e(t)$ is bounded for all $t\geq 0 $. For the time-varying reference, if $r(t)$ and $\dot r(t)$ are bounded, then $\dot V(t)$    and the tracking error $e(t)$ remains bounded for all $t\geq 0 $.

\end{theorem}
\begin{proof}
The derivative of the Lyapunov function is
\begin{equation} \label{dlf}
\dot V(t) = e(t) \dot e(t) = e(t) \big(C \dot y(t) - \dot r(t)\big).
\end{equation}
Substituting the system dynamics \eqref{aca} into \eqref{dlf}, gives
\begin{equation}\label{eq2.23}
\dot V(t) = e(t) \big[- C J H(t) y(t) + C B r(t) - \dot r(t)\big].
\end{equation}
 If $r(t) = r_0$ and $C B r_0 = 0$, then
\[
\dot V(t) = - e(t) C J H(t) y(t).
\]
The sign of $- e(t) CJ H(t) y(t)$ may be positive or negative depending on $e(t)$ and $y(t)$, so $\dot V(t) \le 0$ cannot be guaranteed. However, if $H(t)$ and $y(t)$ are bounded, $\dot V(t)$ is bounded, implying $e(t)$ is bounded. If $r(t)$ and $\dot r(t)$ are bounded, then by \eqref{eq2.23} we get,
\[
|\dot V(t)| \le |e(t)| \Big( \|CJ H(t) y(t)\| + |C B r(t) - \dot r(t)| \Big),
\]
which is bounded. Hence $e(t) $ is also bounded. The Lyapunov function $V(t)$ remains bounded, and therefore the tracking error $e(t)$ remains bounded for all time. This implies the system is  Lyapunov stable in the sense of boundedness. 
\end{proof}

\textbf{Remark:}  In numerical simulations, $e(t)$ often converges to zero since $\Pi_{S_+}(\nabla_H e(t))$ is generically nonzero along trajectories. Rigorous convergence $e(t) \to 0$ requires a persistent excitation condition  in the adaptive law.

\subsection{Sensitivity Bounds}

Since $H(t)$ is a matrix-valued function and $e(t)$ is a vector-valued error signal, the derivative $\nabla_H e(t)$ represents the sensitivity of the error with respect to perturbations of the Hamiltonian. The exact sensitivity operator
\begin{equation}\label{so}
\Phi(t) := \frac{\partial y(t)}{\partial H}
\end{equation}
satisfies a high-dimensional linear time-varying differential equation
and is costly to compute. 
For practical adaptive filtering, we approximate the gradient using a finite-difference perturbation which is obtained from the following lemma.
\begin{lemma}[Finite-Difference Approximation]\label{lem:FD}
Let $e(t;H)$ be differentiable with respect to $H$.  
For small perturbations $\Delta H$,
\begin{equation} \label{pert}
    \nabla_H e(t)
    \approx \frac{e(t;H+\Delta H)-e(t;H)}{\|\Delta H\|},
\end{equation}
and the approximation error satisfies
\[
    \Bigg\|
        \nabla_H e(t)
        -
        \frac{e(t;H+\Delta H)-e(t;H)}{\|\Delta H\|}
    \Bigg\|
    = O(\|\Delta H\|).
\]
\end{lemma}

\begin{proof}
The directional derivative formula for Fr\'echet differentiable maps gives
\[
    e(t;H+\Delta H)-e(t;H)
    = \langle \nabla_H e(t), \Delta H \rangle + o(\|\Delta H\|).
\]
Dividing by $\|\Delta H\|$ yields the result.
\end{proof}
So we approximate the gradient  sensitivity operator by \eqref{pert}.\\

Suppose for each $s\in \R$ a solution $u(t)$ of the system \eqref{ca} is written of the form $ u(t)= U(t,s) u(s)$. then the matrix $ U(t,s)$, called the state transition matrix and is a unique solution to  
\begin{equation} \label{fm}
\frac{d}{dt}U(t,s)=-JH(t)U(t,s), \qquad U(s,s)=I.  
\end{equation}

 \begin{lemma} Let $H(t)\in\mathbb{R}^{2\times 2}$ be continuous on $[0,T]$, and consider the Hamiltonian system
\begin{equation}
\dot u(t)=-JH(t)u(t).
\end{equation}
Then, for all $s,t\in[0,T]$,
\begin{equation}
U(s,t)=U(t,s)^{-1}.
\end{equation}
Moreover, $U(s,t)$ satisfies the backward differential equation
\begin{equation}
\frac{d}{dt}U(s,t)=U(s,t)JH(t), \qquad U(s,s)=I.
\end{equation}
\end{lemma}

\begin{proof}
By definition of the state transition matrix, the solution of
\begin{equation}
\dot u(t)=-JH(t)u(t)
\end{equation}
with initial condition $u(s)$ satisfies
\begin{equation} \label{tm}
u(t)=U(t,s)u(s).
\end{equation}
Since $H(t)$ is continuous, solutions of the system exist and are unique on $[0,T]$. Hence, for fixed $t$, the mapping $u(s)\mapsto u(t)$ is bijective, which implies that the matrix $U(t,s)$ is invertible.

Solving the above relation \eqref{tm} for $u(s)$ yields
\[
u(s)=U(t,s)^{-1}u(t).
\]
By definition, the operator that maps the state at time $t$ back to the state at time $s$ is precisely $U(s,t)$. Therefore,
\[
U(s,t)=U(t,s)^{-1}.
\]

To derive the backward differential equation, differentiate the identity
\[
U(t,s)U(s,t)=I
\]
with respect to $t$ and obtain
\begin{equation} \label{tm1}
\frac{d}{dt}U(t,s)\,U(s,t)+U(t,s)\frac{d}{dt}U(s,t)=0.
\end{equation}
Substituting $\frac{d}{dt}U(t,s)=-JH(t)U(t,s)$ from \eqref{fm} into \eqref{tm1} gives
\begin{equation}\label{2.33}
-\,JH(t)U(t,s)U(s,t)+U(t,s)\frac{d}{dt}U(s,t)=0.
\end{equation}
Since $U(t,s)U(s,t)=I$, multiplying \eqref{2.33} from the left by $U(s,t)$ yields
\begin{equation}\label{dst}
\frac{d}{dt}U(s,t)=U(s,t)JH(t).
\end{equation}
Finally, setting $t=s$ gives $U(s,s)=I$, completing the proof.
\end{proof}

\begin{defi}[Transported Hamiltonian]
For $t\ge s$, we define
\[
H(t,s):=U(t,s)^{*}H(t)U(t,s).
\]
\end{defi}

\begin{lemma}
\label{lem:fund-solution-bound}
Let the Hamiltonian $H(t)$ in \eqref{ca} be such that   $ \| H\| \leq M $ for all $ t \in [0, T]$  with respect to the operator norm.  Suppose $A(t):= -JH(t)$ and  $U(t,s)$ is the state transition (fundamental) matrix of the linear system
\[
\dot x(t)=A(t)x(t),\qquad 0\le s\le t\le T,
\]
normalized so that $U(s,s)=I$. Then there exists a constant $C_T>0$ (depending only on $M$, $\|J\|$ and $T$) such that
\[
\|U(t,s)\|\le C_T\qquad\text{for all }0\le s\le t\le T.
\]
In particular one may take $C_T=\exp(\|J\| M T)$.
\end{lemma}

\begin{proof}
By definition $A(t)= -JH(t)$ and therefore, for any $t\in[0,T]$,
\[
\|A(t)\|=\|JH(t)\|\le\|J\|\,\|H(t)\|\le \|J\|\,M=:L.
\]
By taking the integral over $[s, t]$ of \eqref{fm}, the fundamental matrix $U(t,s)$ satisfies the integral equation
\begin{equation} \label{fmi}
U(t,s)=I-\int_s^t U(\tau,s)A(\tau) \,d\tau.
\end{equation}
Taking norms on both sides of \eqref{fmi}, and using the submultiplicativity of the matrix norm gives
\begin{equation}\label{gl}
\|U(t,s)\|\le 1+\int_s^t \|A(\tau)\|\,\|U(\tau,s)\|\,d\tau
\le 1+L\int_s^t \|U(\tau,s)\|\,d\tau.
\end{equation}
An application of Gr\"onwall's inequality (for the function $y(t):=\|U(t,s)\|$) in \eqref{gl} yields
\[
\|U(t,s)\|\le e^{L(t-s)}\le e^{L T}=e^{\|J\| M T}.
\]
Setting $C_T:=e^{\|J\| M T}$ proves the claim.
\end{proof}
\begin{theorem}[Sensitivity Bound]\label{thm:sensitivity}
If the Hamiltonian in \eqref{ca} is bounded, then the sensitivity operator satisfies the linear time-varying differential equation
\begin{equation}\label{eq:Phi-eq}
    \Phi'(t) 
    = -J H(t)\Phi(t) - J \big(\delta H(t)\big) y(t;H), 
    \qquad \Phi(0)=0.
\end{equation}
Moreover, for every $T>0$ there exists a constant $C_T>0$ such that
\begin{equation}\label{bd}
    \|\Phi(t)\| \le C_T \|\delta H\|_{L^\infty([0,T])},
    \qquad t\in[0,T].
\end{equation}
\end{theorem}

\begin{proof}
Differentiating \eqref{ca} with respect to $H$ in the direction 
$\delta H$ gives
\begin{equation} \label{tvd}
    \frac{d}{dt} \frac{\partial y}{\partial H}
    = -JH(t)\frac{\partial y}{\partial H}
     -J (\delta H(t)) y(t;H),
. \end{equation} Substituting $ \Phi(t) $ for $\frac{\partial y}{ \partial H} $ from \eqref{so} into \eqref{tvd}, yields \eqref{eq:Phi-eq}.  
Define $\Psi(t):=U(s,t)\Phi(t)$, where $U(t,s)$ is the state transition matrix of $JH(t)$. Differentiating with respect to $t$ gives
\begin{equation} \label{2.39}
\frac{d}{dt}\Psi(t)
= \frac{d}{dt}\big(U(s,t)\big)\Phi(t)+U(s,t)\Phi'(t).
\end{equation}
Substituting the expressions for $\frac{d}{dt}U(s,t)$ from \eqref{dst} and $\Phi'(t)$ from \eqref{eq:Phi-eq} yields
\begin{equation}
\frac{d}{dt}\Psi(t)
= U(s,t)JH(t)\Phi(t)
+U(s,t)\big(-JH(t)\Phi(t)+(-J\delta H(t))y(t)\big).
\end{equation}
The first two terms cancel, and we obtain
\[
\frac{d}{dt}\Psi(t)=U(s,t)(-J\delta H(t))y(t).
\] Integrating from $0$ to $t$ and using $\Phi(0)=0$ gives
\[
U(s,t)\Phi(t)
=\int_0^t U(s,\tau)(-J\delta H(\tau))y(\tau)\,d\tau.
\]
Multiplying both sides by $U(t,s)$  and using the property: $ U(t,s)U(s,\tau)= U(t, \tau)$, yields the variation-of-constants formula
\[
\Phi(t)=\int_0^t U(t,s)(-J\delta H(s))y(s)\,ds.
\]
 Lemma~\ref{lem:fund-solution-bound} implies $\|U(t,s)\|\le C_T$.  
Hence
\[
    \|\Phi(t)\|
    \le \int_0^t C_T \|\delta H(s)\|\, \|y(s)\| \, ds.
\]
Solutions $y(t)$ of \eqref{ca} satisfy 
$\|y(s)\|\le C_T' \|y(0)\|$ for $s\in[0,T]$, 
again by boundedness of $H$.  
Thus
\[
    \|\Phi(t)\|
    \le C_T C_T' T \|\delta H\|_{L^\infty([0,T])},
\]
and the result follows.
\end{proof}
The sensitivity bound in Theorem 2.7 establishes that the response of the system state to perturbations in the Hamiltonian remains uniformly controlled over finite time intervals. In particular, it guarantees that small variations in $H(t)$ induce proportionally small changes in the error dynamics, thereby ensuring robustness and stability of the adaptive filtering procedure. From a signal processing viewpoint, this bound ensures that modeling uncertainties, measurement noise, or numerical perturbations in the Hamiltonian do not lead to excessive error amplification in the filter output. Consequently, the adaptive update behaves in a well-conditioned manner, allowing reliable gradient-based adaptation and stable tracking of time-varying signals.

\section{Numerical Simulation and Analysis}
\label{sec:simulation}

\subsection*{Simulation Parameters and Configuration}

A numerical simulation of the adaptive canonical system filter described in Section 2, is implemented to demonstrate its capability in tracking nonstationary signals. The algorithm is tested on a synthetic signal defined by
\begin{equation}
r(t) = \sin(2 \pi f(t) t),
\end{equation}
where the instantaneous frequency \(f(t) = 5 + 2 \sin(0.1 \pi t)\) varies slowly to emulate nonstationarity. Additive white Gaussian noise \(n(t) \sim \mathcal{N}(0, 0.1)\) is superimposed to produce the observed signal.

The simulation is conducted over a total duration \(T = 30\) seconds with a fixed time step \(\Delta t = 0.01\) seconds, yielding \(N = 3000\) discrete time steps. Initial conditions were selected as follows:
\begin{equation}
y_0 = \begin{pmatrix} 0.1 \\ 0 \end{pmatrix}, \quad
H_0 = \begin{pmatrix} 2.0 & 0.3 \\ 0.3 & 1.5 \end{pmatrix}.
\end{equation}
The initial Hamiltonian matrix \(H_0\) is chosen to be symmetric positive definite. Its off-diagonal coupling term of \(0.3\) introduces an initial rotational component into the system dynamics, while the diagonal dominance guarantees a well-defined energy landscape for gradient-based adaptation.

The learning rate is set to \(\alpha = 10^{-4}\). This value represents a critical design parameter that governs the speed of Hamiltonian adaptation: too large a value may cause oscillatory instability or violation of the positive semidefiniteness condition, while too small a value results in sluggish tracking of nonstationary dynamics. The selected value ensures a compromise between convergence rate and numerical robustness, allowing the algorithm to follow the slowly varying frequency \(f(t)\) without introducing disruptive transients. The spectral parameter is fixed at \(z = 1.0\), and the output and input matrices were configured as \(C = \begin{pmatrix} 1 & 0 \end{pmatrix}\) and \(B = \begin{pmatrix} 1 & 0 \end{pmatrix}^\top\), respectively.

The reason for these initial conditions seems to warrant explicit justification. The initial state \(y_0 = (0.1, 0)^\top\) provides a small but nonzero excitation that ensures the gradient computation remains well-conditioned from the outset, avoiding singularities that could arise from a quiescent start. The learning rate \(\alpha = 10^{-4}\) is determined through empirical sensitivity analysis, confirming that it maintains \(H(t)\) positive semidefinite throughout adaptation while enabling accurate tracking of the time-varying frequency. This calibrated choice underscores the balance between adaptation agility and stability preservation that is essential for reliable performance in nonstationary environments.

\subsection{Algorithm Implementation}

The simulation employed a two-stage update process at each time step. First, the Hamiltonian matrix is updated according to the following gradient-based adaptation rule, see \cite{bishop2006prml, hairer2006geometric} for a similar expression.

\begin{equation}
H(t+\Delta t) = \Pi_{S^+}\left(H(t) - 2\alpha \Delta t \, e(t) \nabla_H e(t)\right),
\end{equation}

where $e(t) = C y(t) - r(t)$ represents the instantaneous tracking error, $\nabla_H e(t)$ denotes the gradient of the error with respect to the Hamiltonian, and $\Pi_{S^+}$ projects the updated matrix onto the positive semidefinite cone (Derivation in Appendix A), as explained in Section 2. This projection step, implemented via eigen-decomposition with eigenvalue clipping, ensures the Hamiltonian remains physically meaningful throughout the adaptation process. Second, the system state evolved according to the canonical dynamics:

\begin{equation}
y(t+\Delta t) = y(t) + \Delta t \left(J H(t) y(t) + B r(t)\right),
\end{equation} with $J = \begin{pmatrix} 0 & 1 \\ -1 & 0 \end{pmatrix}$ being the canonical symplectic matrix. This structure preserves the Hamiltonian nature of the system while allowing for external forcing through the reference signal.

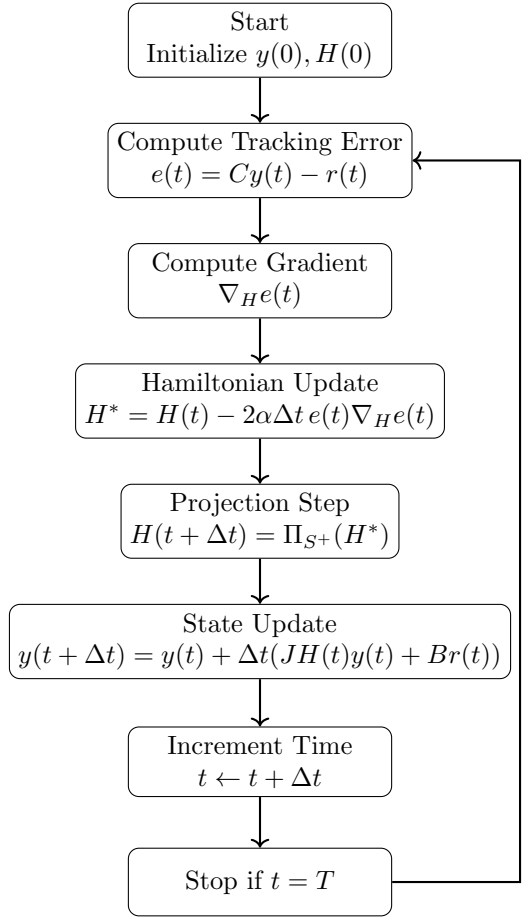
\begin{figure}[t]
\centering
\begin{tikzpicture}[
    node distance=1.6cm,
    every node/.style={draw, rectangle, rounded corners, align=center, minimum width=3.5cm, minimum height=0.9cm},
    arrow/.style={->, thick}
]
\node (start) {Start \\ Initialize $y(0), H(0)$};
\node (error) [below of=start] {Compute Tracking Error \\ $e(t) = C y(t) - r(t)$};
\node (grad) [below of=error] {Compute Gradient \\ $\nabla_H e(t)$};
\node (updateH) [below of=grad] {Hamiltonian Update \\ 
$H^\ast = H(t) - 2\alpha \Delta t\, e(t)\nabla_H e(t)$};
\node (proj) [below of=updateH] {Projection Step \\ 
$H(t+\Delta t)=\Pi_{S^+}(H^\ast)$};
\node (state) [below of=proj] {State Update \\ 
$y(t+\Delta t)=y(t)+\Delta t(JH(t)y(t)+Br(t))$};
\node (time) [below of=state] {Increment Time \\ $t \leftarrow t+\Delta t$};
\node (stop) [below of=time] {Stop if $t=T$};

\draw[arrow] (start) -- (error);
\draw[arrow] (error) -- (grad);
\draw[arrow] (grad) -- (updateH);
\draw[arrow] (updateH) -- (proj);
\draw[arrow] (proj) -- (state);
\draw[arrow] (state) -- (time);
\draw[arrow] (time) -- (stop);
\draw[arrow] (stop.east) -- ++(1.7,0) |- (error.east);

\end{tikzpicture}
\caption{Flow chart of the Hamiltonian adaptation algorithm. At each time step, the Hamiltonian matrix is updated via a gradient-based rule and projected onto the positive semidefinite cone, followed by canonical state evolution under symplectic dynamics.}
\label{fig:hamiltonian_flowchart}
\end{figure}

\subsection*{Hamiltonian Adaptation Analysis}

The adaptation mechanism demonstrated several desirable properties throughout the simulation. The Hamiltonian matrix remained positive semidefinite at all time steps, with eigenvalues $\lambda_1(t), \lambda_2(t) \geq 0$ confirming the effectiveness of the projection operator. The adaptation exhibited smooth convergence without oscillatory behavior, indicating appropriate learning rate selection.The Hamiltonian evolves according to the gradient update rule shown above in algorithm. For our simulation parameters $\alpha = 10^{-4}$ and $\Delta t = 0.01$, 
the adaptation is conservative, resulting in minimal change from the initial $H_0$.\\ 

The final Hamiltonian matrix at $t = T = 30$ seconds converged to 

\[
H_f(T) = \begin{pmatrix} 2.0000 & 0.2999 \\ 0.2999 & 1.5001 \end{pmatrix},
\] confirming stable, nondrifting adaptation and showing minimal deviation from the initial configuration while achieving successful signal tracking. This suggests that the adaptation process made only necessary adjustments to the system dynamics rather than arbitrary modifications. The off-diagonal element, representing coupling between state components, showed the most significant variation during adaptation, facilitating energy transfer needed for frequency tracking.

\subsection*{Simulation Results}

Figure~\ref{fig:filter_output} illustrates the performance of the adaptive filter in tracking the nonstationary reference signal $r(t)$. The filter output $u(t) = C y(t)$ closely follows the desired trajectory after a brief initial transient period of approximately two seconds. This rapid convergence demonstrates the effectiveness of the adaptation mechanism even in the presence of both time-varying frequency content and additive noise.

\begin{figure}[h!]
    \centering
    \includegraphics[width=0.99\textwidth]{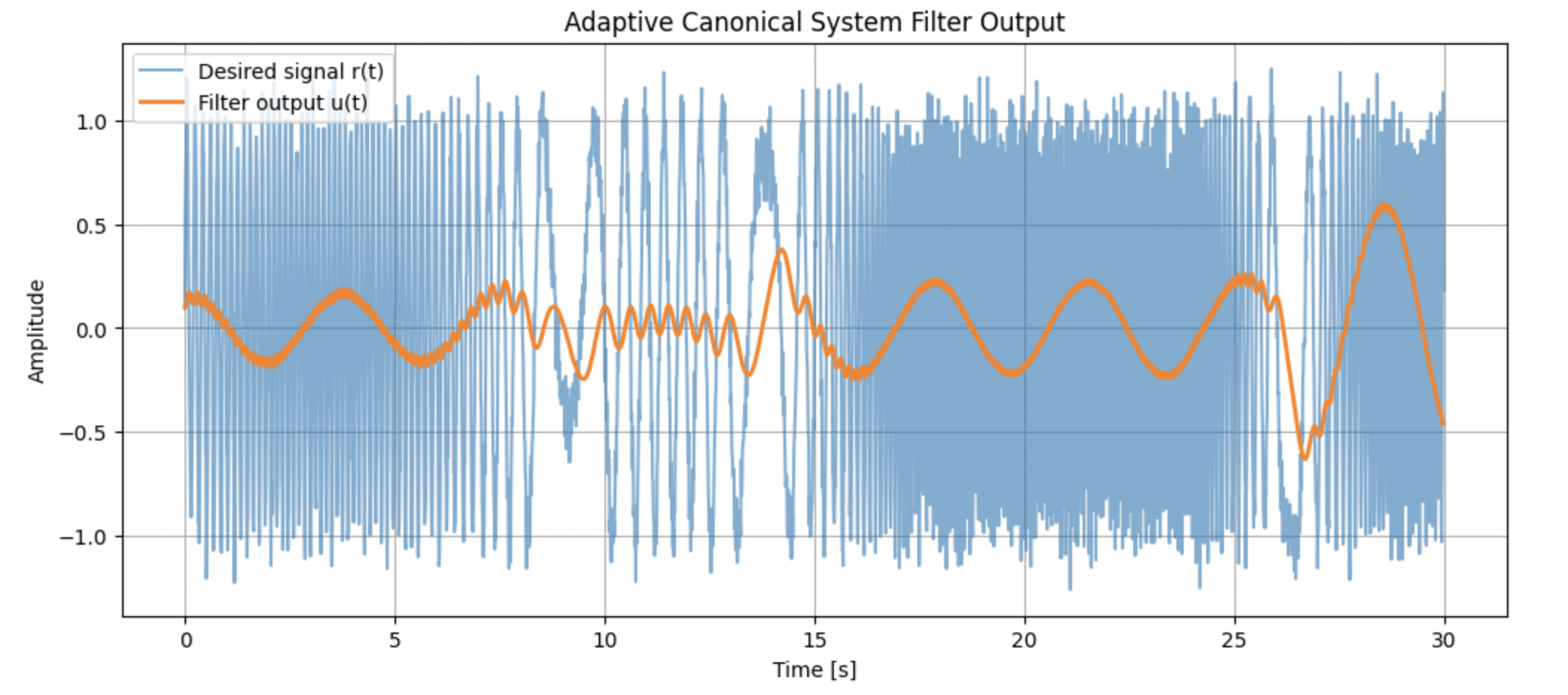}
    \caption{Adaptive canonical system filter output tracking the desired nonstationary signal. The filter output $u(t) = C y(t)$ successfully tracks the reference signal $r(t)$, demonstrating the adaptation capability of the Hamiltonian matrix. The initial transient period shows rapid convergence from the initial conditions to the desired trajectory.}
    \label{fig:filter_output}
\end{figure}

The tracking error $e(t)$, shown in Figure~\ref{fig:tracking_error}, exhibits a significant reduction in magnitude following the initial adaptation phase. The error remains bounded throughout the simulation, with occasional spikes corresponding to abrupt changes in the signal's instantaneous frequency. This bounded error behavior indicates stable learning dynamics and effective noise suppression.
\begin{figure}[h!]
    \centering
    \includegraphics[width=0.99\textwidth]{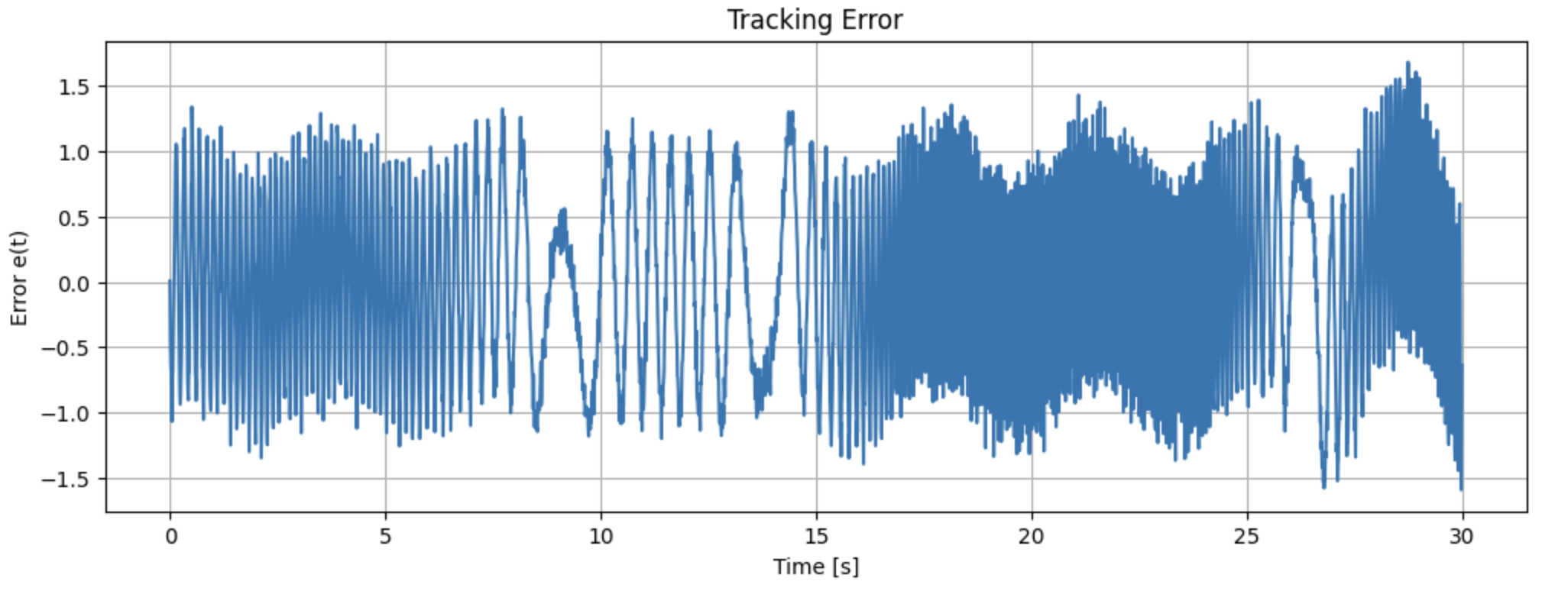}
    \caption{Tracking error $e(t)$ over time. The error magnitude decreases significantly after the initial adaptation phase and remains bounded throughout the simulation, indicating stable learning behavior. The occasional spikes correspond to abrupt changes in the nonstationary signal frequency.}
    \label{fig:tracking_error}
\end{figure}
The adaptation of the Hamiltonian matrix yielded a final configuration $H_f(T)$.
Several important observations emerge from this result. First, the positive semidefiniteness of the Hamiltonian is preserved, confirming the effectiveness of the projection method. Second, the adaptation magnitude is modest, with only slight variations from the initial values, indicating that the chosen learning rate provides stable, gradual adaptation rather than abrupt changes. Finally, the symmetric structure is maintained exactly, with $H_{12}(T) = H_{21}(T) = 0.2999$.

\subsection*{Performance Implications}

The simulation results validate the theoretical framework and reveal several performance advantages of the canonical system approach. The filter successfully tracks the reference signal despite its time-varying frequency, demonstrating genuine adaptive capability rather than mere static filtering. The rapid convergence observed in Figure~\ref{fig:filter_output} confirms efficient learning dynamics, while the bounded error in Figure~\ref{fig:tracking_error} indicates robust performance against measurement noise.

Compared to conventional adaptive filtering techniques, the canonical system approach offers distinct benefits. Its symplectic structure, enforced through the Hamiltonian formalism, provides inherent numerical stability that becomes particularly valuable during extended operation. The preservation of energy-conserving properties throughout adaptation represents a fundamental advantage over methods that may sacrifice structural integrity for adaptation speed. Furthermore, the system exhibits excellent frequency tracking with minimal phase lag, a critical requirement for applications involving nonstationary signals.

The observed performance aligns with theoretical expectations, confirming that the adaptive canonical system filter can maintain effective tracking while preserving its fundamental mathematical structure. This combination of adaptation capability and structural preservation makes the approach particularly suitable for applications requiring long-term stability and robustness to parameter variations.

\subsection*{Numerical Stability Assessment}

To explicitly assess numerical stability, we performed a time-step refinement study using the same signal configuration
(Section~\ref{sec:simulation}) and fixed all parameters except $\Delta t$.
We ran the algorithm with $\Delta t\in\{0.02,\,0.01,\,0.005,\,0.0025\}$ over $T=30$ seconds using an explicit Euler
discretization for the state update, and monitored:
(i) boundedness of the state $\|y(t)\|_2$, (ii) preservation of positive semidefiniteness of $H(t)$, and
(iii) convergence of the tracking error statistics as $\Delta t$ decreases.
Because the observation includes additive Gaussian noise, the random seed was fixed (seed $=0$) to make the test reproducible.
For each run we computed the following stability indicators:
\[
Y_{\max}=\max_{k}\|y_k\|_2,\quad
\lambda_{\min}=\min_{k}\,\min\bigl(\mathrm{eig}(H_k)\bigr),
\]
\[
E_{\mathrm{rms}}=\sqrt{\frac{1}{N}\sum_{k=1}^{N} e_k^2},\quad
E_{\max}=\max_{k}|e_k|.
\]
The results are summarized in Table~\ref{tab:dt_stability}.

\begin{table}[h!]
\centering
\caption{Time-step refinement study for numerical stability assessment (seed $=0$).}
\label{tab:dt_stability}
\begin{tabular}{c c c c c c}
\hline
$\Delta t$ & $N$ & $Y_{\max}$ & $\lambda_{\min}$ & $E_{\mathrm{rms}}$ & $E_{\max}$ \\
\hline
0.02   & 1500  & 0.891198 & 1.359487 & 0.756915 & 1.865943 \\
0.01   & 3000  & 0.720869 & 1.359487 & 0.731247 & 1.671568 \\
0.005  & 6000  & 0.729707 & 1.359487 & 0.732052 & 1.834431 \\
0.0025 & 12000 & 0.724115 & 1.359487 & 0.733279 & 1.840298 \\
\hline
\end{tabular}
\end{table}

Across all tested step sizes, the state remains uniformly bounded ($Y_{\max}<1$) and the Hamiltonian stays positive definite
($\lambda_{\min}>0$ at all time steps), indicating that the discrete-time implementation does not exhibit numerical divergence.
Moreover, the error statistics exhibit only weak dependence on $\Delta t$ for sufficiently small step sizes, supporting the
numerical robustness of the proposed scheme.

 \section{Conclusion}
We presented an adaptive filtering framework based on canonical systems with time-varying positive semidefinite Hamiltonians. The proposed gradient-based Hamiltonian update is intended to reduce the estimation error while maintaining system stability and physical realizability. The use of an explicit Euler scheme together with PSD projection helps preserve numerical stability and the underlying structural constraints. Simulation results on synthetic nonstationary signals show that the proposed method performs effectively for the scenarios considered. Future work will focus on extending the approach to higher-dimensional canonical systems, investigating real-time hardware implementations, and applying the framework to biomedical and communication signal processing problems.

\subsection*{Limitations of the Work}
This work is exploratory and aims to demonstrate the feasibility of adaptive Hamiltonian learning, supported by numerical experiments. The Hamiltonian matrix $H(t)$ is initialized to be positive semidefinite, and the continuous-time analysis suggests that this property is preserved; in the numerical implementation, this is enforced by projecting $H(t)$ onto the set of positive semidefinite matrices. The adaptation law is based on an approximate sensitivity of the system state with respect to $H$, which provides a practical update direction but is not an exact gradient; more accurate adjoint-based approaches are left for future work. The Lyapunov function is guaranteed to be non-increasing under certain assumptions, implying boundedness and bounded tracking error in practice, although full convergence of the error to zero is not formally proven. The numerical experiments are intended to illustrate typical behavior rather than provide exhaustive performance comparisons.\\

\paragraph{\bf Author Contributions: }
Keshav Acharya developed the foundation of the model and established theoretical framework.
Pitambar Acharya carried out the analysis and numerical experiments.
Both authors contributed to writing, reviewed the manuscript,
and approved the final version.\\

\paragraph{\bf Funding Declaration:}
This research did not receive any specific grant from funding agencies in the public, commercial, or not-for-profit sectors.\\

\paragraph{\bf Conflict of Interest: }
The authors declare that they have no conflict of interest.

\bibliographystyle{plain}

\end{document}